\documentclass[10pt]{amsart}
\title[boundaries]{Invariant boundary distributions for finite graphs}
\author{Guyan Robertson}
\address{School of Mathematics and Statistics, University of Newcastle, NE1 7RU, U.K.}
\email{a.g.robertson@newcastle.ac.uk}

\usepackage{amsmath}
\usepackage{amscd}
\usepackage{amssymb}

\font\fiverm=cmr5
\input{prepictex}
\input{pictex}
\input{postpictex}
\chardef\bslash=`\\ 

\makeatletter
\def\verbatim{\interlinepenalty\@M \@verbatim
  \leftskip\@totalleftmargin\advance\leftskip2pc
  \frenchspacing\@vobeyspaces \@xverbatim}
\makeatother
\newtheorem{theorem}{Theorem}[section]
\newtheorem{corollary}[theorem]{Corollary}
\newtheorem{lemma}[theorem]{Lemma}
\newtheorem{proposition}[theorem]{Proposition}

\theoremstyle{definition}
\newtheorem{example}[theorem]{Example}

\newtheorem{remark}[theorem]{Remark}

\newcommand{\ovl}{\overline}
\newcommand{\cl}[1]{{\mathcal{#1}}}
\newcommand{\bb}[1]{{\mathbb{#1}}}
\newcommand{\fk}[1]{{\mathfrak{#1}}}

\newcounter{picture}

\newcommand{\vf}{{\varphi}}

\newcommand{\SL}{{\text{\rm{SL}}}}

\begin{document}

\begin{abstract} Let $\Gamma$ be the fundamental group of a finite connected graph $\cl G$. Let $\fk M$ be an abelian group. A {\it distribution} on the boundary $\partial\Delta$ of the universal covering tree $\Delta$ is an $\fk M$-valued measure defined on clopen sets.
If $\fk M$ has no $\chi(\cl G)$-torsion then the group of $\Gamma$-invariant distributions on $\partial\Delta$ is isomorphic to $H_1(\cl G,\fk M)$.
\end{abstract}

\maketitle

\section{Introduction}

Let $\cl G$ be a finite connected graph, and let $\Delta$ be its universal covering tree. The edges of $\Delta$ are directed and each geometric edge corresponds to two directed edges $\delta$ and $\ovl \delta$.
Let $\Delta ^0$ be the set of vertices and $\Delta^1$  the set of directed edges of $\Delta$. The boundary $\partial\Delta$ is the set of equivalence classes of infinite semi-geodesics in $\Delta$, where two semi-geodesics are equivalent if they agree except on finitely many edges. The boundary has a natural compact totally disconnected topology. The fundamental group $\Gamma$ of $\cl G$ is a free group which acts on $\Delta$ and on $\partial \Delta$. If $\fk M$ is an abelian group then an $\fk M$-valued {\it distribution} on $\partial\Delta$ is a finitely additive $\fk M$-valued measure $\mu$ defined on the  clopen subsets of $\partial \Delta$.
By integration, a distribution may be regarded as an $\fk M$-linear function on the group $C^\infty(\partial \Delta, \fk M)$ of locally constant $\fk M$-valued functions on $\partial \Delta$.
Let
$\fk D^\Gamma(\partial\Delta, \fk M)$ be the additive group of all $\Gamma$-invariant $\fk M$-valued distributions on $\partial\Delta$ and let
$$\fk D^\Gamma_0(\partial\Delta, \fk M)= \{\mu\in \fk D^\Gamma(\partial\Delta, \fk M) : \mu(\partial\Delta)=0\}.$$

\begin{theorem}\label{main}
  There is an isomorphism of abelian groups $$H_1(\cl G, \fk M)\cong \fk D^\Gamma_0(\partial\Delta,\fk M).$$
\end{theorem}
Let $\chi(\cl G)$ denote the Euler characteristic of $\cl G$. If $\fk M$ has no $\chi(\cl G)$-torsion then each element of $\fk D^\Gamma(\partial\Delta, \fk M)$ has total mass zero [Proposition \ref{L3}].
Setting $\fk M=\bb Z$ gives:
\begin{corollary}
    There is an isomorphism of abelian groups $$H_1(\cl G, \bb Z)\cong \fk D^\Gamma(\partial\Delta,\bb Z).$$
\end{corollary}

A. Haefliger and L. Banghe have proved a continuous analogue of Theorem \ref{main}: if $\Gamma$ is the fundamental group of a compact surface of genus $g$ then the space of $\Gamma$-invariant (classical) distributions on $S^1$ which vanish on constant functions has dimension $2g$  \cite[Theorem 5.A.2]{lott}.

The motivation for this article came from $C^*$-algebraic K-theory, as explained in Section \ref{K-theory} below.

\section{Construction of distributions}

Choose an orientation on $\Delta^1$ which is invariant under $\Gamma$. This orientation consists of a partition of $\Delta^1$ and a bijective involution
$$\delta\mapsto \overline \delta : \Delta^1 \to \Delta^1$$
 which interchanges the two components of $\Delta^1$. Each directed edge $\delta$ has an
initial vertex $o(\delta)$ and a terminal vertex $t(\delta)$, such that $o(\overline \delta)=t(\delta)$. The maps $\delta\mapsto \overline \delta$, $\delta\mapsto o(\delta)$ and $\delta\mapsto t(\delta)$ are $\Gamma$-equivariant.
The quotient graph $\cl G=\Gamma\backslash\Delta$ has vertex set $V=\Gamma\backslash\Delta^0$
and directed edge set $E=\Gamma\backslash\Delta^1$.
There are induced maps $x\mapsto \overline x$, $x\mapsto o(x)$ and $x\mapsto t(x)$ on the quotient and
the partition of $\Delta^1$ passes to a partition
$$E=E_+\sqcup\overline{E_+}.$$

If $\delta\in \Delta^1$, let $\Omega_\delta$ be the clopen subset of $\partial\Delta$ corresponding to the set of all semi-geodesics with initial edge $\delta$ and initial vertex $o(\delta)$.

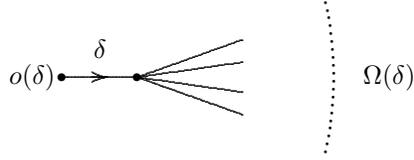
\centerline{
\beginpicture
\setcoordinatesystem units <1cm, 1cm>    
\setplotarea  x from 0 to 4,  y from -1.5 to 1.5
\put{${\delta}$}   [b] at 0.5 0.25
\put{$_\bullet$}  at  0 0
\put{$o(\delta)$}  at  -0.4 0
\put{$_\bullet$}  at  1 0
\put{${\Omega(\delta)}$}   [l] at 4 0
\arrow <6pt> [.2, .67] from  0.4 0  to 0.6 0
\setlinear
\plot    0  0  1 0  2.4 0.5 /
\plot    1  0  2.4 0.2 /
\plot    1  0  2.4 -0.2 /
\plot    1  0  2.4 -0.5 /
\setplotsymbol({$\cdot$}) \plotsymbolspacing=3pt
\circulararc 25 degrees from 3.5 -1 center at -1 0
\endpicture
}
The sets $\Omega_\delta$, $\delta\in \Delta^1$ form a basis for a totally disconnected compact topology on $\partial \Delta$ which is described as an inverse limit in \cite[I.2.2]{ser}. Any clopen set $V$ in $\partial\Delta$ is a finite disjoint union of sets of the form $\Omega_\delta$. Indeed, choose a base vertex $\xi$. Then, for all sufficiently large $n$, $V$ is a disjoint union of sets of the form $\Omega_\delta$, where $\delta$ is directed away from $o$ and $d(o(\delta), \xi) = n$. The following relations are satisfied :
\begin{equation}\label{relation}
 \Omega_\delta= \displaystyle\bigsqcup_{\substack{o(\delta')=t(\delta) \\ \delta'\ne\overline\delta}}\Omega_{\delta'}.
\end{equation}
Let $\fk M$ be an abelian group and let $\mu$ be an $\fk M$-valued distribution on $\partial\Delta$. Then
\begin{subequations}\label{rel1}
\begin{eqnarray}
\sum_{\substack{\delta\in \Delta^1 \\ o(\delta)=\xi}}\mu(\Omega_\delta) &=& \mu(\partial\Delta)  ,\qquad \text{for} \ \xi\in \Delta^0; \label{rel1a}\\
\mu(\Omega_\delta) + \mu(\Omega_{\overline \delta}) &=& \mu(\partial\Delta)   ,\qquad \text{for} \ \delta\in \Delta^1. \label{rel1b}
\end{eqnarray}
\end{subequations}
\refstepcounter{picture}
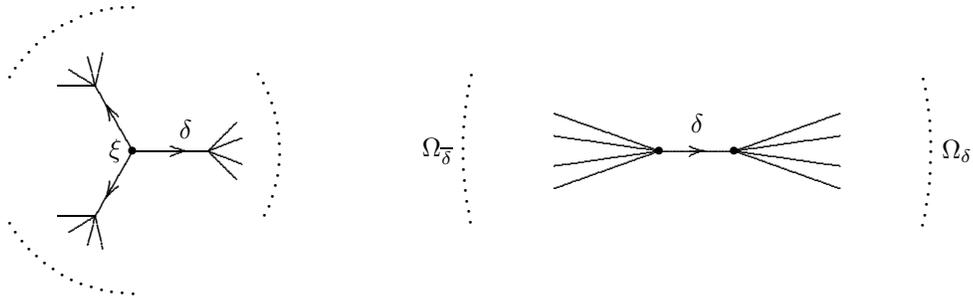
\begin{figure}[htbp]\label{relationdiag}
\centerline{
\beginpicture
\setcoordinatesystem units <0.5cm,0.866cm>   
\setplotarea x from -4 to 4, y from -2  to 2         
\putrule from 0 0 to 2 0
\putrule from -2 1 to -1 1
\putrule from -2 -1 to -1 -1
\setlinear
\plot -1 1  0 0  -1 -1 /
\plot 2.75 0.433  2 0  2.75 -0.433  /
\plot 2.9 0.2  2 0  2.9 -0.2  /
\plot -1.7 1.25  -1 1  -1.2  1.45  /
\plot -1.7 -1.25  -1 -1  -1.2  -1.45  /
\plot  -1 1   -0.8 1.5 /
\plot  -1 -1   -0.8 -1.5 /
\put{$_\bullet$}  at  0 0
\put{$\delta$}[b] at  1.4 0.2
\put{$\xi$}[r] at  -0.3 0
\arrow <6pt> [.3,.67] from  1.2 0 to  1.4 0
\arrow <6pt> [.3,.67] from  -0.6 0.6 to  -0.7 0.7
\arrow <6pt> [.3,.67] from  -0.6 -0.6 to -0.7 -0.7
\setplotsymbol({$\cdot$}) \plotsymbolspacing=4pt
\circulararc 60 degrees from 3.5 -1 center at 0 0
\circulararc 60 degrees from 0 2.2 center at 0 0
\circulararc -60 degrees from 0 -2.2 center at 0 0
\setplotsymbol({\fiverm .}) \plotsymbolspacing= .4pt  
\setcoordinatesystem units <1cm, 1cm> point at -7 0
\setplotarea  x from -4 to 4,  y from -2 to 2
\put{$\delta$}   [b] at 0.5 0.25
\put{$_\bullet$}  at  0 0
\put{$_\bullet$}  at  1 0
\put{${\Omega_\delta}$}   [l] at 3.75 0
\put{${\Omega_{\overline \delta}}$}   [r] at -2.75 0
\arrow <6pt> [.2, .67] from  0.4 0  to 0.6 0
\setlinear
\plot   -1.4 0.5  0 0  1 0  2.4 0.5 /
\plot    1  0  2.4 0.2 /
\plot    1  0  2.4 -0.2 /
\plot    1  0  2.4 -0.5 /
\plot    0  0  -1.4 0.2 /
\plot    0  0  -1.4 -0.2 /
\plot    0  0  -1.4 -0.5 /
\setplotsymbol({$\cdot$}) \plotsymbolspacing=4pt
\circulararc 25 degrees from 3.5 -1 center at -1 0
\circulararc 25 degrees from -2.5 1 center at 2 0
\endpicture
}
\end{figure}

For each $\alpha = \sum_{x\in E_+} n_xx \in H_1(\cl G, \fk M)$, define a $\Gamma$-invariant distribution $\mu_\alpha$ by
\begin{equation}\label{map}
\mu_\alpha(\Omega_\delta)= \langle\alpha-\overline\alpha, \Gamma \delta\rangle ,
\end{equation}
where $\langle\cdot\,, \cdot\rangle$ is the standard inner product $\fk ME \times \fk ME \to \fk M$ and $\overline\alpha = \sum_{x\in E_+} n_x\overline x$. Thus, if $x\in E_+$,
$$
\mu_\alpha(\Omega_\delta) =
\begin{cases}
n_x & \text{if \quad $\Gamma \delta=x$},\\
-n_x & \text{if \quad $\Gamma \delta=\ovl x$}.
\end{cases}
$$
The verification that $\mu_\alpha$ is well defined is given below and it is clear from (\ref{map}) that $\mu_\alpha$ is $\Gamma$-invariant.
 \begin{example}
For the directed graph below, with two vertices and three edges, the boundary distribution corresponding to the $1$-cycle $\alpha = a+2b-3c$ satisfies
$$
\mu_\alpha(\Omega_\delta) =
\begin{cases}
1 & \text{if \quad $\Gamma \delta=a$},\\
2 & \text{if \quad $\Gamma \delta=b$} \\
-3 & \text{if \quad $\Gamma \delta=c$}.
\end{cases}
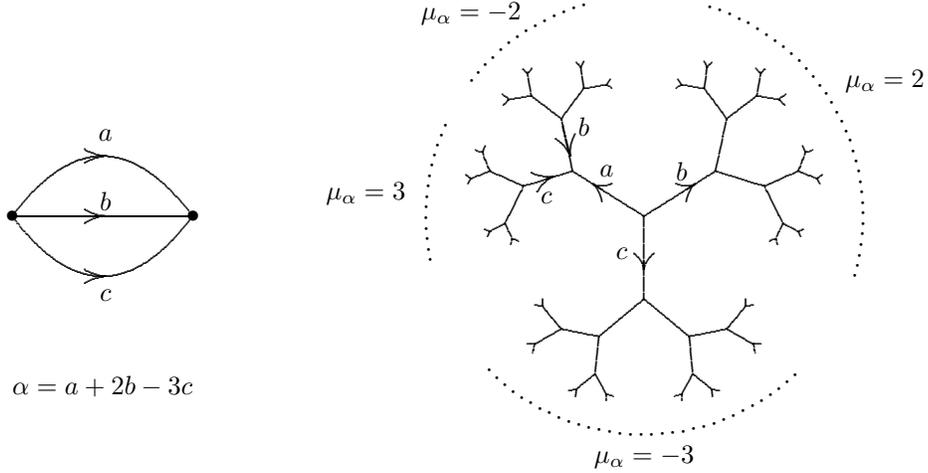
$$
\refstepcounter{picture}
\begin{figure}[htbp]\label{3loop}
\hfil
\centerline{
\beginpicture
\setcoordinatesystem units <.1cm,.1cm>    
\setplotarea x from -25 to 25, y from -32  to 32         
\setlinear
\plot  0 0  9.5 6 / \plot  9.5 6   16 4 / \plot  9.5 6   11 13 /
\plot 16 4  18.5 -1 / \plot 16 4  20.5 6 / \plot 11 13   15 16 / \plot 11 13  8 17 /
\plot 18.5 -1  20.5 -2 / \plot 18.5 -1  17.5 -3 /  \plot 20.5 6  23 5 / \plot 20.5 6  21.5 8.5 /
\plot 15 16  18 15.5  / \plot 15 16  15.5 19  /  \plot 8 17  8.4 19.8 /   \plot 8 17  5 17.5 /
\plot  20.5 -2   20.8 -2.7 /  \plot  20.5 -2   21.2 -1.7 / \plot 17.5 -3 16.7 -3.3 / \plot 17.5 -3 17.7 -3.7 /
\plot 23 5  23.7 5.5 /  \plot 23 5 23.4 4.3 / \plot 21.5 8.5  22.4 8.8 /  \plot 21.5 8.5  21.1 9.4 /
\plot 18 15.5  18.7 14.8 /\plot 18 15.6  18.8 16.2  / \plot 15.5 19  14.9 19.7  / \plot 15.5 19  16.1 19.6 /
\plot 8.4 19.8  9 20.5 / \plot 8.4 19.8  8 20.6 /
\plot 5 17.5 4.5 18.2 / \plot 5 17.5 4.3 17.1 /
\arrow <6pt> [.4, 1.3] from  5.7 3.6   to   6.72 4.2
\arrow <6pt> [.4, 1.3] from  -5.7 3.6   to   -6.72 4.2
\arrow <6pt> [.4, 1.3] from  0 -6   to   0  -7
\arrow <8pt> [.3, 1] from  -10.8 12  to  -9.8 8  
\arrow <8pt> [.3, 1] from   -13  5  to  -12  5.3 
\put{$a$}   at  -5  6
\put{$b$}   at   5  6
\put{$c$}   at  -3 -5
\put{$b$}   at   -8 12
\put{$c$}   at   -13 2.5
\put{$\mu_\alpha=2$}  at  32 18
\put{$\mu_\alpha=-2$}  at  -23 27
\put{$\mu_\alpha=3$}  at  -37 3
\put{$\mu_\alpha=-3$}  at  0 -32
\plot  0 0  -9.5 6 / \plot  -9.5 6   -16 4 / \plot  -9.5 6   -11 13 /
\plot -16 4  -18.5 -1 / \plot -16 4  -20.5 6 / \plot -11 13   -15 16 / \plot -11 13  -8 17 /
\plot -18.5 -1  -20.5 -2 / \plot -18.5 -1  -17.5 -3 /  \plot -20.5 6  -23 5 / \plot -20.5 6  -21.5 8.5 /
\plot -15 16  -18 15.5  / \plot -15 16  -15.5 19  /  \plot -8 17  -8.4 19.8 /   \plot -8 17  -5 17.5 /
\plot  -20.5 -2   -20.8 -2.7 /  \plot  -20.5 -2   -21.2 -1.7 / \plot -17.5 -3 -16.7 -3.3 / \plot -17.5 -3  -17.7 -3.7 /
\plot -23 5  -23.7 5.5 /  \plot -23 5 -23.4 4.3 / \plot -21.5 8.5  -22.4 8.8 /  \plot -21.5 8.5   -21.1 9.4 /
\plot -18 15.5  -18.7 14.8 /\plot -18 15.6  -18.8 16.2  / \plot -15.5 19  -14.9 19.7  / \plot -15.5 19  -16.1 19.6 /
\plot -8.4 19.8  -9 20.5 / \plot -8.4 19.8  -8 20.6 /
\plot -5 17.5 -4.5 18.2 / \plot -5 17.5 -4.3 17.1 /
\plot  0 0   0 -11 /
\plot  0 -11  6 -16 /
\plot  6 -16  11 -15 /
\plot  11 -15  13.5 -12 / \plot  11 -15   14.5 -17 /
\plot  6 -16  6.5 -21 /
\plot  6.5 -21   9 -23 / \plot  6.5 -21   5 -23.7 /
\plot  13.5 -12  13.5 -11 /  \plot  13.5 -12  14.5 -11.8 / \plot  14.5 -17  15.5 -17 /  \plot  14.5 -17  14.8 -18.1 /
\plot 9 -23  10 -23.2 / \plot 9 -23  9.1 -24 /  \plot 5 -23.7   5 -24.5 / \plot 5 -23.7   4.1 -23.9 /
\plot  0 -11  -6 -16 /
\plot  -6 -16  -11 -15 /
\plot  -11 -15  -13.5 -12 / \plot  -11 -15   -14.5 -17 /
\plot  -6 -16  -6.5 -21 /
\plot  -6.5 -21   -9 -23 / \plot  -6.5 -21   -5 -23.7 /
\plot  -13.5 -12  -13.5 -11 /  \plot  -13.5 -12  -14.5 -11.8 / \plot  -14.5 -17  -15.5 -17 /  \plot  -14.5 -17  -14.8 -18.1 /
\plot -9 -23  -10 -23.2 / \plot -9 -23  -9.1 -24 /  \plot -5 -23.7   -5 -24.5 / \plot -5 -23.7   -4.1 -23.9 /
\setplotsymbol({$\cdot$}) \plotsymbolspacing=4pt
\circulararc 90 degrees from 28 -8 center at 0 0
\circulararc 37 degrees from -8  28 center at 0 0
\circulararc 37 degrees from -26.4  12 center at 0 0
\circulararc 90 degrees from -20.6  -20.6 center at 0 0
\setcoordinatesystem units <0.4cm, 0.4cm> point at 18 0
\setplotarea  x from -4 to 4,  y from -2 to 2
\setplotsymbol({\fiverm .}) \plotsymbolspacing= .4pt  
\put{$\alpha = a+2b-3c$}   [b] at 0 -6
\put{$a$}   [b] at 0.1 2.5
\put{$b$}   [b] at 0.1 0.2
\put{$c$}   [b] at 0.1 -2.8
\put{$\bullet$}  at -3 0
\put{$\bullet$}  at 3 0
\arrow <8pt> [.2, .67] from  0 0  to  0.1 0
\arrow <8pt> [.2, .67] from  0 2  to  0.1 2
\arrow <8pt> [.2, .67] from  0 -2  to  0.1 -2
\putrule from -3 0 to 3 0
\setlinear
\setquadratic
\plot -3 0  0 2  3 0 /
\plot -3 0  0 -2  3 0 /
\endpicture
}
\end{figure}
\end{example}

Recall that $H_1(\cl G, \fk M)=\ker\partial$, where the boundary map  $\partial : \fk M E_+ \to \fk M V$ is defined by
$\partial x = t(x)-o(x)$ \cite[Section II.2.8]{ser}.
In order to check that $\mu_\alpha$ is well defined, if $\alpha\in H_1(\cl G, \fk M)$, it is enough to show that equation (\ref{map}) respects the relation (\ref{relation}). We must therefore show that, for all $x\in E$,
\begin{equation}
\langle\alpha-\overline\alpha, Tx\rangle =  \langle\alpha-\overline\alpha, x\rangle ,
\end{equation}
where $T : \fk M E \to \fk M E$ is defined by
\begin{equation}\label{T}
Tx = \sum_{\substack{o(y)=t(x)\\ y\ne \ovl x}}y= \left(\sum_{o(y)=t(x)}y\right) -\overline x.
\end{equation}
Equivalently, it is necessary that $(I-T^*)(\alpha-\overline\alpha)=0$, where $T^*$ is the adjoint of $T$.
Define homomorphisms $\vf_0: \fk MV \to \fk ME$ and $\vf_1: \fk ME_+ \to \fk ME$ by
\begin{equation*}
  \vf_0(v) = \sum_{o(y)=v} y \qquad \text{and} \qquad   \vf_1(x)=x-\ovl x.
\end{equation*}
An easy calculation, using the identity
\begin{equation*}
(I-T^*)x =x+\ovl x- \left(\sum_{t(y)=o(x)}y\right),
\end{equation*}
shows that the following diagram commutes:
  \begin{equation}\label{cd}
\begin{CD}
\fk MV                   @<\partial<<                  \fk ME_+\\
@V\vf_0VV                                              @VV\vf_1V \\
\fk ME                  @<<I-T^*<                        \fk ME
\end{CD}
\end{equation}
If $\alpha \in H_1(\cl G, \fk M)$, it follows that
$(I-T^*)(\alpha-\overline\alpha)=\varphi_0\circ \partial (\alpha)=0$,
as required.

\begin{lemma}
 Let $\fk M$ be an abelian group.
  The map $\alpha\mapsto\mu_\alpha$ is an injection from $H_1(\cl G, \fk M)$ into $\fk D^\Gamma_0(\partial\Delta,\fk M)$.
\end{lemma}

\begin{proof}
 Fix $\alpha = \sum_{x\in E_+} n_xx \in H_1(\cl G, \fk M)$. Choose $\delta\in\Delta^1_+$ and let $y=\Gamma\delta\in E_+$. It follows from (\ref{rel1b}) that
 \begin{equation*}
\begin{split}
\mu_\alpha(\partial\Delta)  & =\langle\alpha-\overline\alpha, y + \overline y\rangle\\
& =\langle\alpha-\overline\alpha, y\rangle + \langle\alpha-\overline\alpha, \overline y\rangle\\
& =n_y - n_y\\
& = 0.\\
\end{split}
\end{equation*}
Thus $\mu_\alpha\in \fk D^\Gamma_0(\partial\Delta,\fk M)$.
The proof that the map is injective is straightforward.
\end{proof}

\begin{remark}
  The map $\alpha\mapsto\mu_\alpha$ clearly depends on the choice of orientation on the tree $\Delta$, although the group $H_1(\cl G, \fk M)$ does not \cite[Section II.2.8]{ser}.
\end{remark}

The next result completes the proof of the Theorem \ref{main}.

\begin{lemma}\label{L2}
  Let $\fk M$ be an abelian group.
     The map $\alpha\mapsto \mu_\alpha$ is a surjection from
  $H_1(\cl G, \fk M)$ onto $\fk D^\Gamma_0(\partial\Delta,\fk M)$.
\end{lemma}

\begin{proof}
  Let $\mu\in \fk D^\Gamma_0(\partial\Delta,\fk M)$. Since $\mu$ is $\Gamma$-invariant, we may define a function
  $\lambda : E \to \fk M$ by $\lambda(\Gamma\delta)= \mu(\Omega_\delta)$.
  Let $\alpha=\sum_{x\in E_+} \lambda(x)x$. We show that $\alpha\in H_1(\cl G, \fk M)$ and that $\mu=\mu_\alpha$.

Since $\mu(\partial\Delta)=0$, the relations (\ref{rel1}) project to  the following flow relations on $\cl G$.
\begin{subequations}\label{relK2}
\begin{eqnarray}
\sum_{\substack{x\in E \\ o(x)=v}}\lambda(x) &=& 0  ,\qquad \text{for} \ v\in V; \label{relK2a}\\
{\lambda(x)+\lambda(\overline x)} &=& 0  ,\qquad \text{for} \ x \in E. \label{relK2b}
\end{eqnarray}
\end{subequations}
It follows from equations (\ref{relK2}) that
\begin{equation*}
\begin{split}
\partial\alpha  & =\sum_{x\in E_+}\lambda(x)(t(x)-o(x))\\
& = \sum_{x\in \overline E_+}\lambda(\overline x)t(\overline x)- \sum_{x\in E_+}\lambda(x)o(x)\\
& = -\sum_{x\in \overline E_+}\lambda(x)o(x)- \sum_{x\in E_+}\lambda(x)o(x)\\
& = - \sum_{x\in E}\lambda(x)o(x)\\
& = - \sum_{v\in V}\left(\sum_{\substack{x\in E \\ o(x)=v}}\lambda(x)\right)v = 0.\\
\end{split}
\end{equation*}
Therefore $\alpha\in H_1(\cl G, \fk M)$. Finally, it follows from (\ref{relK2b}) that
$$\alpha-\overline\alpha=\sum_{x\in E} \lambda(x)x$$
and, for each $\delta \in \Delta^1$,
$$\mu_\alpha(\Omega_\delta)= \langle\alpha-\overline\alpha, \Gamma \delta\rangle =\lambda(\Gamma \delta)=\mu(\Omega_\delta).$$
Therefore $\mu=\mu_\alpha$, as required.
\end{proof}

\begin{remark}
A special case occurs when $\bb K$ is a non-archimedean local field and $\Gamma$ is a torsion free cocompact lattice in $\SL(2,\bb K)$. The Bruhat-Tits building associated with $\SL(2,\bb K)$ is a regular tree $\Delta$ whose boundary $\partial\Delta$ may be identified with the projective line $\bb P_1(\bb K)$ \cite[II.1.1]{ser}. In this context, the relations (\ref{relK2}) assert that if $\mu\in \fk D^\Gamma_0(\partial\Delta,\fk M)$ then the function $\delta\mapsto \mu(\Omega_\delta)$ is a $\Gamma$-invariant harmonic cocycle on $\Delta^1$, in the sense of \cite[3.15]{ga}. The relationship between harmonic cocycles and boundary distributions has been studied in the $p$-adic case in \cite{t}.
\end{remark}

In many cases, every $\Gamma$-invariant distribution on $\partial\Delta$ has total mass zero.

\begin{proposition}\label{L3}
  Let $\fk M$ be an abelian group. If $\fk M$ does not have $\chi(\cl G)$-torsion then
   $$\fk D^\Gamma(\partial\Delta,\fk M)=\fk D^\Gamma_0(\partial\Delta,\fk M).$$
\end{proposition}

\begin{proof}
  Let $\mu\in \fk D^\Gamma(\partial\Delta,\fk M)$. For $x\in E$, define $\lambda(x)= \mu(\Omega_\delta)$ if $x=\Gamma\delta$. This is well defined, since $\mu$ is $\Gamma$-invariant.
Let $\sigma=\mu(\partial\Delta)$.
The relations (\ref{rel1}) project to the quotient graph $\cl G$ as follows.
\begin{subequations}\label{relK1}
\begin{eqnarray}
\sum_{\substack{x\in E \\ o(x)=v}}\lambda(x) &=& \sigma  ,\qquad \text{for} \ v\in V; \label{relK1a}\\
{\lambda(x)+\lambda(\overline x)} &=& \sigma  ,\qquad \text{for} \ x \in E. \label{relK1b}
\end{eqnarray}
\end{subequations}

Let $n_0$ [$n_1$] be the number of vertices [edges] of $\cl G$, so that $\chi(\cl G)=n_0-n_1$.
Since the map $x\mapsto o(x)\, :E\to V$ is surjective, the relations (\ref{relK1}) imply that
\begin{equation*}
\begin{split}
n_0\sigma & = \sum_{v\in V}\sum_{\substack{x\in E \\ o(x)=v}}\lambda(x)= \sum_{x\in E}\lambda(x)\\
& = \sum_{x\in E_+}(\lambda(x)+\lambda(\ovl x))= \sum_{x \in E_+}\sigma\\
& = n_1\sigma.
\end{split}
\end{equation*}
Therefore $\chi(\cl G).\sigma=0$. The hypothesis on $\fk M$ implies that $\sigma=0$; in other words, $\mu\in \fk D^\Gamma_0(\partial\Delta,\fk M)$.
\end{proof}

The following example of a nonzero $\Gamma$-invariant boundary distribution shows that the assumption that $\fk M$ does not have $\chi(\cl G)$-torsion cannot be removed.
\begin{example}
  Let $\cl G$ be a $(q+1)$-regular graph, where $q > 3$, so that $\chi(\cl G)=\frac{n_0}{2}(1-q)$, where $n_0$ is the number of vertices of $\cl G$. Let $\fk M = \bb Z_{q-1}$ and define $\mu\in \fk D^\Gamma(\partial\Delta,\bb Z_{q-1})$ by $\mu(\Omega_\delta) = 1$, for all $\delta\in \Delta^1$. Then $\mu(\partial\Delta)=2\not= 0$.
\end{example}

\section{The relation to K-theory}\label{K-theory}

The motivation for this article came from the study of the K-theory of the crossed product $C^*$-algebra $C(\partial\Delta)\rtimes\Gamma$. Suppose that each vertex of $\cl G$ has at least three neighbours. Then the compact space $\partial\Delta$ is perfect (hence uncountable) and $\cl A = C(\partial\Delta)\rtimes\Gamma$ is a Cuntz-Krieger algebra \cite{rtree}. The group $K_1(\cl A)$ is isomorphic to $U(\cl A)/U_0(\cl A)$, the quotient of the unitary group of $\cl A$ by the connected component of the identity, and it follows from \cite{c2} that $K_1(\cl A) \cong \ker(T-1)$, where $T$ is the map defined by equation (\ref{T}), with $\fk M = \bb Z$. We have the following result.

\begin{proposition}\label{k} Suppose that each vertex of $\cl G$ has at least three neighbours. Then the map $\alpha\mapsto \alpha-\overline \alpha$ is an isomorphism from $H_1(\cl G, \bb Z)$ onto  $\ker(T-1)$.
\end{proposition}

\begin{proof}
Let $\alpha=\sum_{x\in E} \lambda(x)x \in \ker(T-I)$.
If $y\in E$, then the coefficient of $y$ in the sum representing $(T-I)\alpha$ is
$$
\left(\sum_{\substack{x\in E, x\ne \ovl y\\ t(x)=o(y)}}\lambda(x)\right) - \lambda(y)
=\left(\sum_{\substack{x\in E \\ t(x)=o(y)}}\lambda(x)\right) - \lambda(y)-\lambda(\ovl y).
$$
This coefficient is zero, since $\alpha\in\ker (T-I)$. Therefore
\begin{equation}\label{l}
  \lambda(y)+\lambda(\ovl y)=\sum_{\substack{x\in E \\ t(x)=o(y)}}\lambda(x).
\end{equation}
For any $y\in E$, define
$\sigma(y)= \lambda(y)+\lambda(\ovl y)$.
The right hand side of equation (\ref{l}) depends only on $o(y)$, therefore $\sigma(y)$ depends
only on $o(y)$. On the other hand, $\sigma(y)=\sigma(\ovl y)$, so that $\sigma(y)$ depends only on $t(y)$.
Since the graph $\cl G$ is connected, it follows that $\sigma(y)=\sigma$, a constant, for all $y \in E$.
The result follows easily, using the arguments of Proposition \ref{L3} and Lemma \ref{L2}, together with the fact that the Euler characteristic of $\cl G$ is non-zero.
\end{proof}

The natural map from $\Gamma$ into $U(\cl A)$ induces a homomorphism from $\Gamma$ into $K_1(\cl A)$.
The isomorphism $K_1(\cl A) \cong \ker(T-1)$ is described explicitly in \cite[Section 2]{rordam}.
Combining this with Proposition \ref{k} and \cite[Section 1]{rtree}, it is easy to see that the
homomorphism $\Gamma\to K_1(\cl A)$ is surjective.

\end{document}